\documentclass[11pt]{article}

\usepackage{amssymb,amsthm,amsmath,multirow}

\usepackage[a4paper,margin=1in]{geometry}

\newtheorem{theorem}{Theorem}[section]
\newtheorem{lemma}[theorem]{Lemma}
\newtheorem{corollary}[theorem]{Corollary}

\theoremstyle{definition}
\newtheorem{definition}[theorem]{Definition}
\newtheorem{example}[theorem]{Example}

\theoremstyle{remark}
\newtheorem{remark}[theorem]{Remark}

\usepackage{enumitem}
\setlist[enumerate,1]{label=(\arabic*)}
\setlist[enumerate,2]{label=(\roman*)}

\numberwithin{equation}{section}

\usepackage{pgfplots}
\pgfplotsset{compat=1.17}

\usepackage{hyperref}
\usepackage[style=alphabetic,isbn=false,url=false,doi=false,giveninits=true,maxnames=9,sorting=nyt,backend=bibtex]{biblatex} 
\renewbibmacro{in:}{}
\bibliography{satomi-convex-convolution-English.bib} 

\AtEveryBibitem{\clearfield{note}}
\DeclareLabelalphaTemplate{
 \labelelement{
   \field[final]{shorthand}
   \field{label}
   \field[strwidth=3,strside=left,ifnames=1]{labelname}
   \field[strwidth=1,strside=left]{labelname}
 }
 \labelelement{
   \field{year}
 }
}

\title{An inequality for the compositions of convex functions with convolutions and an alternative proof of the Brunn--Minkowski--Kemperman inequality \thanks{The final version of the paper is published in Proc. Steklov Inst. Math. 319 (2022): http://dx.doi.org/10.1134/S0081543822050182}}
\author{Takashi Satomi}
\date{\today} 

\begin{document}

\maketitle

\begin{abstract}

  Let $m(G)$ be the infimum of the volumes of all open subgroups of a unimodular locally compact group $G$.
  Suppose integrable functions $\phi_1 , \phi_2 \colon G \to [0,1]$ satisfy $\| \phi_1 \| \leq \| \phi_2 \|$ and $\| \phi_1 \| + \| \phi_2 \| \leq m (G)$, where $\| \cdot \|$ denotes the $L^1$-norm with respect to a Haar measure $dg$ on $G$.
  We have the following inequality for any convex function $f \colon [0, \| \phi_1 \| ] \to \mathbb{R}$ with $f(0) = 0$:
  \begin{align*}
    \int_{G}^{} f \circ ( \phi_1 * \phi_2 ) (g) dg \leq 2 \int_{0}^{\| \phi_1 \|} f(y) dy + ( \| \phi_2 \| - \| \phi_1 \| ) f( \| \phi_1 \| ).
  \end{align*}
  As a corollary, we have a slightly stronger version of Brunn--Minkowski--Kemperman inequality.
  That is, we have
  \begin{align*}
    \mathrm{vol}_* ( B_1 B_2 ) \geq \mathrm{vol} ( \{ g \in G \mid 1_{B_1} * 1_{B_2} (g) > 0 \} ) \geq \mathrm{vol} (B_1) + \mathrm{vol} (B_2)
  \end{align*}
  for any non-null measurable sets $B_1 , B_2 \subset G$ with $\mathrm{vol} (B_1) + \mathrm{vol} (B_2) \leq m(G)$, where $\mathrm{vol}_*$ denotes the inner measure and $1_B$ the characteristic function of $B$.

\end{abstract}



\section{Introduction}

Let $\mathcal{A}$ be the set of all open subgroups of a unimodular locally compact group $G$, and we set
\begin{align*}
  m (G) := \inf_{G' \in \mathcal{A}} \mathrm{vol} (G')
\end{align*}
($m(G)$ may be $\infty$).
Let $\| \cdot \|$ be the $L^1$-norm with respect to a Haar measure $dg$ on $G$.
For two measurable functions $\phi_1 , \phi_2 \colon G \to \mathbb{R}_{\geq 0}$, we define the convolution $\phi_1 * \phi_2$ as
\begin{align*}
  \phi_1 * \phi_2 (g) := \int_{G}^{} \phi_1 (g') \phi_2 (g'^{-1} g) dg'.
\end{align*}
The main result of this paper is the following theorem.

\begin{theorem}
  \label{thm:convex-convolution-L1-L-infinity}

  Suppose integrable functions $\phi_1 , \phi_2 \colon G \to [0,1]$ on a unimodular locally compact group $G$ satisfy $\| \phi_1 \| \leq \| \phi_2 \|$ and
  \begin{equation}
    \| \phi_1 \| + \| \phi_2 \| \leq m (G). \label{eq:m-large}
  \end{equation}
  Then, for any convex function $f \colon [0, \| \phi_1 \| ] \to \mathbb{R}$ with $f(0) = 0$, one has
  \begin{equation}
    \int_{G}^{} f \circ ( \phi_1 * \phi_2 ) (g) dg
    \leq 2 \int_{0}^{\| \phi_1 \|} f(y) dy + ( \| \phi_2 \| - \| \phi_1 \| ) f( \| \phi_1 \| ). \label{eq:convex-convolution-inequality}
  \end{equation}

\end{theorem}

For example, by using the $L^1$-norm (and $L^\infty$-norm) of $\phi_1$ and $\phi_2$, Theorem \ref{thm:convex-convolution-L1-L-infinity} bounds the $L^p$-norm of the convolution $\phi_1 * \phi_2$ when $f(y) = y^p$ holds (Example \ref{ex:main} \ref{item:main-Lp-small} and \ref{item:main-Lp-large}), and the entropy of $\phi_1 * \phi_2$ from below when $f(y) = y \log y$ (Example \ref{ex:main} \ref{item:main-entropy}).
Actually, Theorem \ref{thm:convex-convolution-L1-L-infinity} is best possible for $G = \mathbb{R}$ (Example \ref{ex:main} \ref{item:main-psi0}).

Here are some other examples of Theorem \ref{thm:convex-convolution-L1-L-infinity}.
When $t \geq 0$ is fixed,
\begin{align}
  f_t (y) :=
  \left\{
  \begin{aligned}
     & y - t &  & \text{if} \; y \geq t \\
     & 0     &  & \text{if} \; y \leq t
  \end{aligned}
  \right. \label{eq:ft-define}
\end{align}
is convex.
Since
\begin{align*}
  2 \int_{0}^{\| \phi_1 \|} f_t (y) dy + ( \| \phi_2 \| - \| \phi_1 \| ) f_t ( \| \phi_1 \| )
  = ( \| \phi_1 \| - t) ( \| \phi_2 \| - t)
\end{align*}
holds for any $0 \leq t \leq \| \phi_1 \|$, we have
\begin{equation}
  \| f_t \circ ( \phi_1 * \phi_2 ) \| \leq (\| \phi_1 \| - t )( \| \phi_2 \| - t ) \label{eq:ft-convolution}
\end{equation}
by applying $f = f_t$ to Theorem \ref{thm:convex-convolution-L1-L-infinity}.
Actually, \eqref{eq:ft-convolution} holds even when the assumption \eqref{eq:m-large} is weakened by
\begin{equation}
  \| \phi_1 \| + \| \phi_2 \| \leq m (G) + t. \label{eq:m-large-weak}
\end{equation}
That is, the following lemma holds.

\begin{lemma}
  \label{lem:ft-convolution-L1-L-infinity}

  One has \eqref{eq:ft-convolution} for any integrable functions $\phi_1 , \phi_2 \colon G \to [0,1]$ on a unimodular locally compact group $G$ and for any $0 \leq t \leq \min \{ \| \phi_1 \| , \| \phi_2 \| \}$ with \eqref{eq:m-large-weak}.

\end{lemma}

For any characteristic functions $\phi_1$ and $\phi_2$ on a compact connected group $G$, Lemma \ref{lem:ft-convolution-L1-L-infinity} was essentially proved by Tao \cite[Theorem 3.2.3]{tao2012spending} and explicitly given by Christ \cite[Section 6]{MR3952692}.
Christ also proved Lemma \ref{lem:ft-convolution-L1-L-infinity} for any characteristic functions $\phi_1$ and $\phi_2$ on $G = \mathbb{R}^n$, and discussed a relation between Lemma \ref{lem:ft-convolution-L1-L-infinity} and the Riesz--Sobolev inequality \cite{MR1574064} \cite{sobolev1938theorem}.

A claim similar to Lemma \ref{lem:ft-convolution-L1-L-infinity} for any cyclic group of prime order was proved by Pollard \cite{MR354517}, and for any finite abelian group by Green and Ruzsa \cite[Proposition 6.1]{MR2166359}.
By using this claim, Green and Ruzsa explicitly computed the maximal cardinality of all sum-free subsets for any finite abelian group \cite[Theorem 1.5]{MR2166359}.

In addition, when $\phi_1$ and $\phi_2$ are characteristic functions, we obtain the following corollary by applying Theorem \ref{thm:convex-convolution-L1-L-infinity} to
\begin{align}
  f (y) :=
  \left\{
  \begin{aligned}
     & 0   &  & \text{if} \; y = 0 \\
     & - 1 &  & \text{if} \; y > 0
  \end{aligned}
  \right. . \label{eq:f-apply-Kemperman}
\end{align}

\begin{corollary}
  \label{cor:Kemperman-strong}

  Let $B_1 , B_2 \subset G$ be measurable sets on a unimodular locally compact group $G$ with $\mathrm{vol} (B_1), \mathrm{vol} (B_2) >0$.
  If $\mathrm{vol} (B_1) + \mathrm{vol} (B_2) \leq m (G)$ holds, then one has
  \begin{align}
    \mathrm{vol} ( \{ g \in G \mid 1_{B_1} * 1_{B_2} (g) > 0 \} )
    \geq \mathrm{vol} (B_1) + \mathrm{vol} (B_2), \label{eq:Kemperman-strong-state}
  \end{align}
  where $1_B$ denotes the characteristic function of $B$.
  In particular,
  \begin{equation}
    \mathrm{vol}_* (B_1 B_2) \geq \mathrm{vol} (B_1) + \mathrm{vol} (B_2) \label{eq:Kemperman}
  \end{equation}
  holds, where $\mathrm{vol}_*$ denotes the inner measure.

\end{corollary}

The Brunn--Minkowski--Kemperman inequality \eqref{eq:Kemperman} follows from \eqref{eq:Kemperman-strong-state} and the fact that $\{ g \in G \mid 1_{B_1} * 1_{B_2} (g) > 0 \} \subset B_1 B_2$.
Corollary \ref{cor:Kemperman-strong} was proved by Brascamp and Lieb for $G=\mathbb{R}^n$ \cite[Theorem A.1]{MR0450480}, and by Tao for any compact connected group $G$ by using Lemma 1.2 \cite[Section 3.2]{tao2012spending}.
This argument is essentially equivalent to applying Theorem \ref{thm:convex-convolution-L1-L-infinity} to \eqref{eq:f-apply-Kemperman}.
The inequality \eqref{eq:Kemperman} was proved by Kemperman for any general unimodular locally compact group $G$, and a number of people for particular cases (see Table \ref{tab:Kemperman-author}).

\begin{table}

  \centering
  \caption{The authors who proved \eqref{eq:Kemperman} for various cases}
  \label{tab:Kemperman-author}
  \begin{tabular}{c|c}
    Unimodular locally compact group $G$        & Authors and references                                  \\
    \hline
    $\mathbb{R}^n$ ($B_1$ and $B_2$ are convex) & Brunn \cite{brunn1887ovale}, Minkowski \cite{MR0249269} \\
    \multirow{2}{*}{$\mathbb{R}^n$}             & Lusternik \cite{lusternik1935brunn},                    \\
                                                & Henstock--Macbeath \cite[Theorem 1]{MR56669}            \\
    $S^1$                                       & Raikov \cite{MR0001776}                                 \\
    $(S^1)^n$                                   & Macbeath \cite[Theorem 1]{MR56670}                      \\
    Compact connected second countable abelian  & Shields \cite{MR72201}                                  \\
    Abelian                                     & Kneser \cite{MR81438}                                   \\
    General                                     & Kemperman \cite[Theorem 1.2]{MR202913}                  \\
    Connected or with $m (G) = \infty$          & Ruzsa \cite[Theorem 2]{MR1173766}                       \\
  \end{tabular}

\end{table}

For $G = \mathbb{R}^n$, \eqref{eq:Kemperman} can be improved to
\begin{align}
  \mathrm{vol}_* (B_1 B_2)^{1/n} \geq \mathrm{vol} (B_1)^{1/n} + \mathrm{vol} (B_2)^{1/n}. \label{eq:Brunn-Minkowski}
\end{align}
For any convex sets $B_1$ and $B_2$ on $G = \mathbb{R}^n$, Brunn proved \eqref{eq:Brunn-Minkowski}, and Minkowski proved that the equality is attained in \eqref{eq:Brunn-Minkowski} if and only if $B_1$ and $B_2$ are equal up to dilatation.
For arbitrary measurable sets $B_1$ and $B_2$ on $G = \mathbb{R}^n$, the inequality \eqref{eq:Brunn-Minkowski} was stated by Lusternik.
However, Lusternik's proof is defective, and it was corrected by Henstock and Macbeath.
A generalization of \eqref{eq:Brunn-Minkowski} for any locally compact group $G$ was studied by McCrudden \cite[Theorem 1.2]{MR233921} and Jing, Tran, and Zhang \cite[Theorem 1.1]{jing2021nonabelian}.

When $G$ is connected or $m (G) = \infty$ holds, Ruzsa gave an alternative proof of \eqref{eq:Kemperman}.
Kemperman's idea is to study the sets $B_1$ and $B_2$ such that the maximum of $\mathrm{vol} (B_1) + \mathrm{vol} (B_2) - \mathrm{vol}_* ( B_1 B_2 )$ is attained.
Ruzsa's idea is to show that $T(I) := \inf_{\mathrm{vol} (B_1) = I} \mathrm{vol}_* ( B_1 B_2 )$ is a concave function for any fixed $B_2$.

For any compact connected abelian group $G$, Tao studied a property of measurable sets for which the equality of Corollary \ref{cor:Kemperman-strong} is almost attained \cite[Theorem 1.3]{MR3920221}.
This result is generalized for disconnected groups by Griesmer \cite[Theorem 1.1]{MR4042161}, and for non-abelian groups by Jing and Tran \cite[Theorem 1.2]{jing2020minimal}.
A sharper result for any compact abelian group is given by Christ and Iliopoulou \cite[Theorem 1.4]{MR4494185}.

Here is the strategy of proof of Theorem \ref{thm:convex-convolution-L1-L-infinity}.
As mentioned above, Lemma \ref{lem:ft-convolution-L1-L-infinity} under the assumption \eqref{eq:m-large} is an example of Theorem \ref{thm:convex-convolution-L1-L-infinity}.
However, we will show Lemma \ref{lem:ft-convolution-L1-L-infinity} without using Theorem \ref{thm:convex-convolution-L1-L-infinity} at first, and obtain Theorem \ref{thm:convex-convolution-L1-L-infinity} by using Lemma \ref{lem:ft-convolution-L1-L-infinity}.
We will show Lemma \ref{lem:ft-convolution-L1-L-infinity} by generalizing the argument of Tao \cite[Theorem 3.2.3]{tao2012spending}, and Theorem \ref{thm:convex-convolution-L1-L-infinity} by approximating $f$ by a linear combination of $f_t$.

The organization of this paper is as follows.
In Section \ref{sec:convex-convolution-example}, we will give some examples and remarks concerning Theorem \ref{thm:convex-convolution-L1-L-infinity}.
Sections \ref{sec:assumption-necessary}-\ref{sec:main-ft-m-finite} are devoted to the proof of Lemma \ref{lem:ft-convolution-L1-L-infinity}.
In Section \ref{sec:assumption-necessary}, we will discuss the necessity of \eqref{eq:m-large}.
In Section \ref{sec:S-property}, we will give a key lemma (Lemma \ref{lem:S-property}) to prove Lemma \ref{lem:ft-convolution-L1-L-infinity}.
By using the arguments in Sections \ref{sec:assumption-necessary} and \ref{sec:S-property}, we will show Lemma \ref{lem:ft-convolution-L1-L-infinity} in the case of $m (G) = \infty$ in Section \ref{sec:main-ft-m-infinite}, and in the case of $m (G) < \infty$ in Section \ref{sec:main-ft-m-finite}.
In Section \ref{sec:f-general}, we will complete the proof of Theorem \ref{thm:convex-convolution-L1-L-infinity} by using Lemma \ref{lem:ft-convolution-L1-L-infinity}.

\section{Some examples and remarks related to Theorem \ref{thm:convex-convolution-L1-L-infinity}}
\label{sec:convex-convolution-example}

In this section, we give some examples and remarks concerning Theorem \ref{thm:convex-convolution-L1-L-infinity} (and related results).
At first, we give some examples of application of Theorem \ref{thm:convex-convolution-L1-L-infinity}.

\begin{example}
  \label{ex:main}

  \begin{enumerate}
    \item \label{item:main-psi0}
          One has $m (G) = \infty$ for $G = \mathbb{R}$.
          The function $\phi_{(I)} \colon \mathbb{R} \to \{ 0, 1 \}$ denotes the characteristic function $\phi_{(I)} := 1_{( -I/2 , I/2)}$ of the open interval $( -I/2 , I/2)$ for $I > 0$.
          Let $\psi_{I_1,I_2} := \phi_{(I_1)} * \phi_{(I_2)}$ for $I_1,I_2 > 0$.
          If $I_1 \leq I_2$, then one has
          \begin{align*}
            \psi_{I_1,I_2} (x) =
            \left\{
            \begin{aligned}
               & I_1      &  & \text{if} \; |x| \leq c         \\
               & c' - |x| &  & \text{if} \; c \leq |x| \leq c' \\
               & 0        &  & \text{if} \; c' \leq |x|
            \end{aligned}
            \right.
            , \quad
            c := \frac{I_2 - I_1}{2}, \quad
            c' := \frac{I_1 + I_2}{2}.
          \end{align*}
          Thus, if $\phi_1 = \phi_{(I_1)}$ and $\phi_2 = \phi_{(I_2)}$, then one has
          \begin{align}
            \int_{\mathbb{R}}^{} f \circ \psi_{I_1,I_2} (x) dx
             & = \int_{-c'}^{-c} f \left( x + c' \right) dx + \int_{- c}^{c} f (I_1) dx + \int_{c}^{c'} f \left( c' - x \right) dx \notag \\
             & = 2 \int_{0}^{I_1} f(y) dy + ( I_2 - I_1 ) f( I_1 ). \label{eq:tilde-culculation}
          \end{align}
          Since $I_1 = \| \phi_1 \|$ and $I_2 = \| \phi_2 \|$, the equality is attained in \eqref{eq:convex-convolution-inequality}.
          That is, \eqref{eq:convex-convolution-inequality} can be expressed as
          \begin{align*}
            \int_{G}^{} f \circ ( \phi_1 * \phi_2 ) (g) dg
            \leq \int_{\mathbb{R}}^{} f \circ \psi_{\| \phi_1 \| , \| \phi_2 \|} (x) dx.
          \end{align*}
          (We note that $\phi_1$ and $\phi_2$ are not the same as $\phi_1 = \phi_{(I_1)}$ and $\phi_2 = \phi_{(I_2)}$ stated before and are already general functions on a general unimodular locally compact group $G$.)
          In particular, \eqref{eq:ft-convolution} can be expressed as
          \begin{equation}
            \| f_t \circ ( \phi_1 * \phi_2 ) \|
            \leq \| f_t \circ \psi_{\| \phi_1 \| , \| \phi_2 \|} \|. \label{eq:ft-convolution-psi0}
          \end{equation}

    \item \label{item:main-Lp-small}
          The function $f(y) := y^p$ is convex for $p \geq 1$.
          Thus, one has
          \begin{align*}
            \int_{G}^{} ( \phi_1 * \phi_2 (g) )^p dg
            \leq 2 \int_{0}^{\| \phi_1 \|} y^p dy + ( \| \phi_2 \| - \| \phi_1 \| ) {\| \phi_1 \|}^p
            = {\| \phi_1 \|}^p \left( \| \phi_2 \| - \frac{(p - 1) \| \phi_1 \|}{p + 1} \right)
          \end{align*}
          by applying Theorem \ref{thm:convex-convolution-L1-L-infinity} in this case.

    \item \label{item:main-Lp-large}
          The function $f(y) := -y^p$ is convex for $0 < p \leq 1$.
          Thus, one has
          \begin{align*}
            \int_{G}^{} ( \phi_1 * \phi_2 (g) )^p dg
            \geq 2 \int_{0}^{\| \phi_1 \|} y^p dy + ( \| \phi_2 \| - \| \phi_1 \| ) {\| \phi_1 \|}^p
            = {\| \phi_1 \|}^p \left( \| \phi_2 \| + \frac{(1 - p) \| \phi_1 \|}{p + 1} \right)
          \end{align*}
          by applying Theorem \ref{thm:convex-convolution-L1-L-infinity} in this case.

    \item \label{item:main-entropy}
          The function $f (y) := y \log y$ is convex.
          Thus, one has
          \begin{align*}
            \int_{G}^{} \phi_1 * \phi_2 (g) \log ( ( \phi_1 * \phi_2 ) (g) ) dg
             & \leq 2 \int_{0}^{\| \phi_1 \|} y \log y dy + ( \| \phi_2 \| - \| \phi_1 \| ) \| \phi_1 \| \log \| \phi_1 \| \\
             & = \| \phi_1 \| \left( \| \phi_2 \| \log \| \phi_1 \| - \frac{\| \phi_1 \|}{2} \right)
          \end{align*}
          by applying Theorem \ref{thm:convex-convolution-L1-L-infinity} in this case.

    \item \label{item:main-ft-example-Fubini}
          Since $f_t (y) = y$ holds for $t = 0$, one has $\| \phi_1 * \phi_2 \| \leq \| \phi_1 \| \cdot \| \phi_2 \|$ by Lemma \ref{lem:ft-convolution-L1-L-infinity}.
          In addition, $- \| \phi_1 * \phi_2 \| \leq - \| \phi_1 \| \| \phi_2 \|$ holds by applying Theorem \ref{thm:convex-convolution-L1-L-infinity} to the convex function $f(y) := -y$.
          Thus, one has
          \begin{equation}
            \| \phi_1 * \phi_2 \| = \| \phi_1 \| \cdot \| \phi_2 \|. \label{eq:Fubini}
          \end{equation}
          By Fubini's theorem, \eqref{eq:Fubini} holds for any $\phi_1 , \phi_2 \colon G \to \mathbb{R}_{\geq 0}$ even when the assumption \eqref{eq:m-large} is not satisfied.
          That is, the following lemma holds.

  \end{enumerate}

\end{example}

\begin{lemma}
  \label{lem:Fubini}

  One has \eqref{eq:Fubini} for any integrable functions $\phi_1 , \phi_2 \colon G \to \mathbb{R}_{\geq 0}$ on a unimodular locally compact group $G$.

\end{lemma}

The equalities are attained in Lemma \ref{lem:ft-convolution-L1-L-infinity} and Example \ref{ex:main} \ref{item:main-Lp-small}, \ref{item:main-Lp-large} and \ref{item:main-entropy} for $G = \mathbb{R}$, $\phi_1 = \phi_{(\| \phi_1 \|)}$ and $\phi_2 = \phi_{(\| \phi_2 \|)}$ by Example \ref{ex:main} \ref{item:main-psi0}.

Next, we define the notion of convex function used in Theorem \ref{thm:convex-convolution-L1-L-infinity}.

\begin{definition}
  \label{def:convex}

  For a function $f \colon A \to \mathbb{R}$ on an interval $A$ and $y_1 , y_2 \in A$ with $y_1 < y_2$, the functions $\tilde{f}_{y_1}^{y_2} , \hat{f}_{y_1}^{y_2} \colon A \to \mathbb{R}$ are defined as
  \begin{align*}
    \tilde{f}_{y_1}^{y_2} (y) & := \frac{(y_2 - y) f(y_1) + (y - y_1) f(y_2)}{y_2 - y_1} &  & \text{and} &
    \hat{f}_{y_1}^{y_2} (y)   & := \tilde{f}_{y_1}^{y_2} (y) - f(y).
  \end{align*}
  A function $f \colon A \to \mathbb{R}$ is called convex if $\hat{f}_{y_1}^{y_2} (y) \geq 0$ (that is, $f (y) \leq \tilde{f}_{y_1}^{y_2} (y)$) for any $y_1, y_2, y \in A$ with $y_1 < y_2$ and $y_1 \leq y \leq y_2$.

\end{definition}

Here are some remarks on Theorem \ref{thm:convex-convolution-L1-L-infinity} and Corollary \ref{cor:Kemperman-strong}.

\begin{remark}
  \label{rem:convex-convolution}

  \begin{enumerate}
    \item \label{item:convex-convolution-well-defined}
          The composition $f \circ ( \phi_1 * \phi_2)$ in Theorem \ref{thm:convex-convolution-L1-L-infinity} is well-defined because
          \begin{equation}
            \phi_1 * \phi_2 (g)
            = \int_{G}^{} \phi_1 (g') \phi_2 (g'^{-1} g) dg'
            \leq \int_{G}^{} \phi_1 (g') dg'
            = \| \phi_1 \| \label{eq:psi-leq-I1}
          \end{equation}
          holds for any $g \in G$.

    \item
          When $f$ is of $C^2$-class, $f$ is convex if and only if $f'' \geq 0$ holds.
          However, Theorem \ref{thm:convex-convolution-L1-L-infinity} does not require $f$ to be differentiable or even continuous (at $0$ and $\| \phi_1 \|$).

    \item \label{item:convex-convolution-connected-component}
          The following three conditions are equivalent for any locally compact group $G$:

          \begin{enumerate}
            \item \label{item:mG-finite-main}
                  One has $m (G) < \infty$;

            \item \label{item:mG-finite-open-compact}
                  The locally compact group $G$ has an open compact subgroup;

            \item \label{item:mG-finite-G0-compact}
                  The identity component $G_0$ of $G$ is compact.

          \end{enumerate}

          The equivalence of \ref{item:mG-finite-main} and \ref{item:mG-finite-open-compact} is obtained by the fact that $\mathrm{vol} (G') < \infty$ holds if and only if $G'$ is compact \cite{MR0005741}.
          We show that \ref{item:mG-finite-open-compact} implies \ref{item:mG-finite-G0-compact}.
          There is an open compact subgroup $G' \subset G$ by \ref{item:mG-finite-open-compact}.
          Since $G_0 \subset G'$ holds by the connectedness of $G_0$, one has \ref{item:mG-finite-G0-compact} by the compactness of $G'$.
          We show that \ref{item:mG-finite-G0-compact} implies \ref{item:mG-finite-open-compact}.
          Since $G / G_0$ is a totally disconnected locally compact group \cite[Theorem 7.3]{MR551496}, there is an open compact subgroup $H \subset G / G_0$ \cite{van1931studien} \cite[Theorem 7.7]{MR551496}.
          Then $\pi^{-1} (H) \subset G$ is an open subgroup, where $\pi \colon G \to G /G_0$ is the projection.
          Since $\pi^{-1} (H)$ is compact by \ref{item:mG-finite-G0-compact} \cite[Note 5.24 (a)]{MR551496}, we obtain \ref{item:mG-finite-open-compact}.

          In addition,
          \begin{align*}
            G_0 = \bigcap_{G' \in \mathcal{A}} G'
          \end{align*}
          holds for any $G$ \cite[Theorem 7.8]{MR551496}.
          Thus, if the equivalent conditions \ref{item:mG-finite-main}, \ref{item:mG-finite-open-compact} and \ref{item:mG-finite-G0-compact} are satisfied, then we have $m (G) = \mathrm{vol} (G_0)$.
          Thus, \eqref{eq:m-large} holds if and only if at least one of the following two conditions is satisfied:

          \begin{enumerate}[label=(\alph*)]
            \item \label{item:m-large-vol-G0}
                  One has $\| \phi_1 \| + \| \phi_2 \| \leq \mathrm{vol} (G_0)$;

            \item \label{item:m-large-G0-non-compact}
                  The identity component $G_0$ is not compact.

          \end{enumerate}

          In particular, if $G_0$ is open (e.g., $G$ is a Lie group), then \eqref{eq:m-large} and \ref{item:m-large-vol-G0} are equivalent.

    \item
          A result similar to Theorem \ref{thm:convex-convolution-L1-L-infinity} holds in the case of $\| \phi_2 \| \leq \| \phi_1 \|$ as follows.
          For integrable functions $\phi_1 , \phi_2 \colon G \to [0,1]$ with $\| \phi_2 \| \leq \| \phi_1 \|$ and \eqref{eq:m-large},
          \begin{equation}
            \phi_1 * \phi_2 (g)
            = \int_{G}^{} \phi_1 (g') \phi_2 (g'^{-1} g) dg'
            \leq \int_{G}^{} \phi_2 (g'^{-1} g) dg'
            = \| \phi_2 \| \label{eq:psi-leq-I2}
          \end{equation}
          holds and hence $f \circ ( \phi_1 * \phi_2)$ is well-defined for any function $f \colon [0, \| \phi_2 \| ] \to \mathbb{R}$.
          Then, one has
          \begin{align*}
            \int_{G}^{} f \circ ( \phi_1 * \phi_2 ) (g) dg \leq 2 \int_{0}^{\| \phi_2 \|} f(y) dy + ( \| \phi_1 \| - \| \phi_2 \| ) f( \| \phi_2 \| )
          \end{align*}
          for any convex function $f \colon [0, \| \phi_2 \| ] \to \mathbb{R}$ with $f(0) = 0$ by the argument similar to Theorem \ref{thm:convex-convolution-L1-L-infinity}.

    \item
          In Corollary \ref{cor:Kemperman-strong}, we set $U := \{ g \in G \mid 1_{B_1} * 1_{B_2} (g) > 0 \}$.
          In fact, there is a case where $\mathrm{vol}_* (B_1 B_2) > \mathrm{vol} ( U )$.
          For example, when $G = \mathbb{R}$, $B_1 = (0,1) \cup \{ 2 \}$ and $B_2 = (0,1)$ hold, we have $B_1 B_2 = (0,3) \setminus \{ 2 \}$ and $U = (0,2)$.
          Since $\mathrm{vol} (B_1) = \mathrm{vol} (B_2) = 1$, $\mathrm{vol}_* (B_1 B_2) = 3$ and $\mathrm{vol} ( U ) = 2$ hold, we have
          \begin{align*}
            \mathrm{vol}_* (B_1 B_2)
            > \mathrm{vol} ( U )
            = \mathrm{vol} (B_1) + \mathrm{vol} (B_2).
          \end{align*}

    \item \label{item:convex-convolution-normalize}

          If $\| \phi_1 \| = 0$, then we have $\phi_1 * \phi_2 = 0$ and hence Theorem \ref{thm:convex-convolution-L1-L-infinity} and Lemma \ref{lem:ft-convolution-L1-L-infinity} hold.
          Thus, it suffices to show Theorem \ref{thm:convex-convolution-L1-L-infinity} and Lemma \ref{lem:ft-convolution-L1-L-infinity} for $\| \phi_1 \| > 0$.
          In this case, one can reduce the proof to the case of $\| \phi_1 \| = 1$ by replacing $dg$ and $f(y)$ with $dg / \| \phi_1 \|$ and $f ( y / \| \phi_1 \| )$, respectively.

  \end{enumerate}

\end{remark}

\section{Necessity of the assumption \texorpdfstring{\eqref{eq:m-large}}{(1.1)}}
\label{sec:assumption-necessary}

In this section, we discuss the necessity of the assumption \eqref{eq:m-large} when $G$ is connected.
That is, if \eqref{eq:m-large} does not hold, then \eqref{eq:ft-convolution} does not hold for any $0 < t < \| \phi_1 \| + \| \phi_2 \| - m (G)$ by the following lemma.

\begin{lemma}
  \label{lem:m-small-counterexample}

  Let $\phi_1 , \phi_2 \colon G \to [0,1]$ be integrable functions on a unimodular locally compact group $G$ with $G = G_0$ (that is, $G$ is connected).

  \begin{enumerate}
    \item \label{item:m-small-counterexample-psi-large}
          One has $\phi_1 * \phi_2 \geq \| \phi_1 \| + \| \phi_2 \| - m(G)$.

    \item \label{item:m-small-counterexample-ft-norm}
          If $\| \phi_1 \| + \| \phi_2 \| - m(G) \geq 0$ holds, then one has $\| f_t \circ ( \phi_1 * \phi_2 ) \| = \| \phi_1 \| \| \phi_2 \| - t m(G)$ for any $0 \leq t \leq \| \phi_1 \| + \| \phi_2 \| - m(G)$.

  \end{enumerate}

\end{lemma}

\begin{proof}
  \begin{enumerate}
    \item
          For any $g \in G$,
          \begin{align*}
            \phi_1 * \phi_2 (g)
            = \int_{G}^{} \phi_1 (g') \phi_2 (g'^{-1} g) dg'.
          \end{align*}
          One has
          \begin{align*}
            \phi_1 (g') \phi_2 (g'^{-1} g)
             & = \phi_1 (g') + \phi_2 (g'^{-1} g) - 1 + (1 - \phi_1 (g') ) ( 1 - \phi_2 (g'^{-1} g)) \\
             & \geq \phi_1 (g') + \phi_2 (g'^{-1} g) - 1
          \end{align*}
          by $\phi_1 (g') , \phi_2 (g'^{-1} g) \leq 1$, and hence
          \begin{align*}
            \int_{G}^{} \phi_1 (g') \phi_2 (g'^{-1} g) dg'
            \geq \int_{G}^{} (\phi_1 (g') + \phi_2 (g'^{-1} g) - 1 ) dg'
            = \| \phi_1 \| + \| \phi_2 \| - \mathrm{vol} (G).
          \end{align*}
          Since $m (G) = \mathrm{vol} (G)$ holds by $G = G_0$ and Remark \ref{rem:convex-convolution} \ref{item:convex-convolution-connected-component}, we obtain
          \begin{align*}
            \phi_1 * \phi_2 (g)
            = \int_{G}^{} \phi_1 (g') \phi_2 (g'^{-1} g) dg'
            \geq \| \phi_1 \| + \| \phi_2 \| - \mathrm{vol} (G)
            = \| \phi_1 \| + \| \phi_2 \| - m (G).
          \end{align*}

    \item
          Since $f_t \circ ( \phi_1 * \phi_2 ) = \phi_1 * \phi_2 - t$ holds by \ref{item:m-small-counterexample-psi-large}, we obtain
          \begin{align*}
            \| f_t \circ ( \phi_1 * \phi_2 ) \|
            = \| \phi_1 * \phi_2 - t \|
            = \| \phi_1 \| \| \phi_2 \| - t m(G)
          \end{align*}
          by Lemma \ref{lem:Fubini}.
          \qedhere
  \end{enumerate}

\end{proof}

By Lemma \ref{lem:m-small-counterexample}, the inequality \eqref{eq:ft-convolution} does not hold for any $0 < t < \| \phi_1 \| + \| \phi_2 \| - m (G)$ when $G$ is connected and \eqref{eq:m-large} does not hold.
In fact,
\begin{align*}
  \| f_t \circ ( \phi_1 * \phi_2 ) \|
  = \| \phi_1 \| \cdot \| \phi_2 \| - t m(G)
  = ( \| \phi_1 \| - t ) ( \| \phi_2 \| - t ) + t ( \| \phi_1 \| + \| \phi_2 \| - m (G) - t )
\end{align*}
holds by Lemma \ref{lem:m-small-counterexample} \ref{item:m-small-counterexample-ft-norm} and hence we have $( \| \phi_1 \| - t ) ( \| \phi_2 \| - t ) < \| f_t \circ ( \phi_1 * \phi_2 ) \|$.

\section{First step of the proof of Lemma \ref{lem:ft-convolution-L1-L-infinity}}
\label{sec:S-property}

In this section, we give the key lemma (Lemma \ref{lem:S-property}) to prove Lemma \ref{lem:ft-convolution-L1-L-infinity}.
This idea is based on the submodularity inequality by Pollard \cite{MR354517}, Ruzsa \cite[Lemma 2.2]{MR1173766}, Green and Ruzsa \cite[Proposition 6.1]{MR2166359} and Tao \cite[Theorem 3.2.3]{tao2012spending} \cite[Lemma 2.1]{MR3082216}.

It suffices to show Lemma \ref{lem:ft-convolution-L1-L-infinity} for $\| \phi_1 \| = 1$ by Remark \ref{rem:convex-convolution} \ref{item:convex-convolution-normalize}.
We fix $\phi_1$ with $\| \phi_1 \| = 1$ and simply write $f$ for $f_t$ in \eqref{eq:ft-define}.
The function $S \colon [0, \mathrm{vol} (G)] \to \mathbb{R}_{\geq 0}$ is defined as
\begin{align*}
  S (I) := \sup_{\phi \in \mathcal{B} (I)} \| \xi_\phi \|,
\end{align*}
where $\mathcal{B} (I) := \{ \phi \colon G \to [0,1] \mid \| \phi \| = I \}$ and $\xi_\phi := f \circ ( \phi_1 * \phi )$.
Lemma \ref{lem:ft-convolution-L1-L-infinity} can be expressed as
\begin{equation}
  S (I) \leq (1-t) (I - t) \label{eq:main-ft-S}
\end{equation}
for any $t \leq I \leq m (G) + t - 1$.
We show the following lemma to prove Lemma \ref{lem:ft-convolution-L1-L-infinity}.

\begin{lemma}
  \label{lem:S-property}

  \begin{enumerate}
    \item \label{item:S-property-Lipschitz}
          One has $0 \leq S(I') - S(I) \leq I'-I$ for any $0 \leq I \leq I' \leq \mathrm{vol} (G)$.
          In particular, $S$ is Lipschitz continuous.

    \item \label{item:S-property-convex}
          The function $S$ is convex on $[0,m (G)]$ (or $[0, \infty)$ in the case of $m (G) = \infty$).

    \item \label{item:S-property-t}
          One has $S(t) =0$.

    \item \label{item:S-property-major}
          One has $S(I) \leq (1-t)I$ for any $0 \leq I \leq \mathrm{vol} (G)$.

  \end{enumerate}

\end{lemma}

Lemma \ref{lem:S-property} \ref{item:S-property-convex} is deduced from the following lemma.

\begin{lemma}
  \label{lem:S-convex-proof}

  Let $I > 0$ and $\phi \in \mathcal{B} (I)$.
  We define $R_g \phi \colon G \to [0,1]$ as $R_g \phi (g') := \phi (g' g)$ for $g \in G$.
  Let $I^2/m(G) < I_1' \leq I$ and $I_2' := 2I - I_1'$.

  \begin{enumerate}
    \item \label{item:S-convex-proof-h}
          We define a continuous function $h \colon G \to \mathbb{R}_{\geq 0}$ as $h(g) := \| \phi R_g \phi \|$.
          If $I_1' \leq h (e)$ holds, then there is $g \in G_0$ with $h (g) = I_1'$, where $e \in G$ is the identity.

    \item \label{item:S-convex-proof-nu-exist}
          There are $g \in G_0$ and $\nu_1 \in \mathcal{B} (I_1')$ such that $\phi R_g \phi \leq \nu_1 \leq \min \{ \phi, R_g \phi \}$ holds.

    \item \label{item:S-convex-proof-nu2}
          Let $g \in G_0$ and $\nu_1 \in \mathcal{B} (I_1')$ be as in \ref{item:S-convex-proof-nu-exist}.
          One has $\max \{ \phi , R_g \phi \} \leq \nu_2 := \phi + R_g \phi - \nu_1 \leq 1$ and $\nu_2 \in \mathcal{B} (I_2')$.

    \item \label{item:S-convex-proof-xi-sum}
          Let $g \in G_0$ and $\nu_1 \in \mathcal{B} (I_1')$ be as in \ref{item:S-convex-proof-nu-exist}, and $\nu_2 \in \mathcal{B} (I_2')$ be as in \ref{item:S-convex-proof-nu2}.
          One has $\xi_\phi + \xi_{R_g \phi} \leq \xi_{\nu_1} + \xi_{\nu_2}$.

    \item \label{item:S-convex-proof-xi-norm}
          Let $\nu_1 \in \mathcal{B} (I_1')$ be as in \ref{item:S-convex-proof-nu-exist}, and $\nu_2 \in \mathcal{B} (I_2')$ be as in \ref{item:S-convex-proof-nu2}.
          One has $2 \| \xi_\phi \| \leq \| \xi_{\nu_1} \| + \| \xi_{\nu_2} \|$.
          In particular, $2 S(I) \leq S(I_1') + S(I_2')$ holds.

  \end{enumerate}

\end{lemma}

\begin{proof}

  \begin{enumerate}
    \item
          Since $h$ is continuous and $G_0$ is connected, it suffices to show that there is $g' \in G_0$ with $h(g') \leq I_1'$.
          We divide the proof into the following two cases.

          \begin{enumerate}
            \item \label{item:S-convex-proof-h-m-finite}
                  The case of $m(G) < \infty$.
                  Then
                  \begin{align*}
                    \| h \|_{G_0}
                    = \int_{G_0}^{} \int_{G}^{} \phi (g'') \phi (g'' g') dg'' dg'
                    \leq I^2
                  \end{align*}
                  holds, where $\| h \|_{G_0}$ denotes the $L^1$-norm of $h$ on $G_0$.
                  Since $\mathrm{vol} (G_0) = m(G)$ by Remark \ref{rem:convex-convolution} \ref{item:convex-convolution-connected-component}, there is $g' \in G_0$ with
                  \begin{align*}
                    h(g') \leq \frac{I^2}{m (G)}
                    \leq I_1'.
                  \end{align*}

            \item \label{item:S-convex-proof-h-m-infinite}
                  The case of $m (G)= \infty$.
                  There is a compact set $K \subset G$ with $\| \phi \|_K \geq I- I_1'/2$.
                  Since $G_0$ is not compact by Remark \ref{rem:convex-convolution} \ref{item:convex-convolution-connected-component}, there is $g' \in G_0 \setminus K^{-1}K$.
                  Then
                  \begin{align*}
                    h (g') = \| \phi R_{g'} \phi \|_K + \| \phi R_{g'} \phi \|_{G \setminus K}
                    \leq \| R_{g'} \phi \|_K + \| \phi \|_{G \setminus K}
                  \end{align*}
                  holds by $\phi , R_{g'} \phi \leq 1$.
                  Since $Kg' \subset G \setminus K$ by $g' \in G_0 \setminus K^{-1}K$, one has
                  \begin{align*}
                    \| R_{g'} \phi \|_K = \| \phi \|_{Kg'} \leq \| \phi \|_{G \setminus K}.
                  \end{align*}
                  Thus, we obtain
                  \begin{align*}
                    h (g')
                    \leq \| R_{g'} \phi \|_K + \| \phi \|_{G \setminus K}
                    \leq 2 \| \phi \|_{G \setminus K}
                    \leq I_1'.
                  \end{align*}

          \end{enumerate}

    \item
          Let $e \in G$ be the identity.
          When $I_1' \leq h (e)$ holds, we obtain
          \begin{align*}
            \nu_1 := \phi R_g \phi \leq \min \{ \phi, R_g \phi \}
          \end{align*}
          by taking $g \in G_0$ in \ref{item:S-convex-proof-h}.
          When $I_1' \geq h (e)$ holds, there is $\nu_1 \in \mathcal{B} (I_1')$ with $\phi^2 \leq \nu_1 \leq \phi$.
          Thus,
          \begin{align*}
            \phi R_g \phi = \phi^2 \leq \nu_1 \leq \min \{ \phi, R_g \phi \} = \phi
          \end{align*}
          holds by taking $g = e$.

    \item
          Since $\nu_1 \leq \min \{ \phi , R_g \phi \}$, one has
          \begin{align*}
            \nu_2 \geq \phi + R_g \phi - \min \{ \phi , R_g \phi \} = \max \{ \phi , R_g \phi \} .
          \end{align*}
          As
          \begin{align*}
            \nu_2
            \leq \phi + R_g \phi - \phi R_g \phi
            = 1 - ( 1 - \phi ) ( 1 - R_g \phi )
          \end{align*}
          holds by $\nu_1 \geq \phi R_g \phi$, one has $\nu_2 \leq 1$ by $\phi , R_g \phi \leq 1$.
          Thus, we obtain $\nu_2 \in \mathcal{B} (I_2')$ by
          \begin{align*}
            \| \nu_2 \| = \| \phi \| + \| R_g \phi \| - \| \nu_1 \| = I_2'.
          \end{align*}

    \item
          We fix $g' \in G$ and define $u ( \nu ) := \phi_1 * \nu (g')$ for an integrable function $\nu \colon G \to [0,1]$.
          Since $\nu_1 \leq \phi \leq \nu_2$ and $\nu_1 \leq R_g \phi \leq \nu_2$ by \ref{item:S-convex-proof-nu-exist} and \ref{item:S-convex-proof-nu2}, one has
          \begin{equation}
            u ( \nu_1 ) \leq u ( \phi ) \leq u ( \nu_2 ) , \quad
            u ( \nu_1 ) \leq u ( R_g \phi ) \leq u ( \nu_2 ). \label{eq:convolution-nu-inequality}
          \end{equation}
          If $u ( \nu_1 ) = u ( \nu_2 )$, then the equalities are attained in \eqref{eq:convolution-nu-inequality} and hence one has $\xi_\phi (g') + \xi_{R_g \phi} (g') = \xi_{\nu_1} (g') + \xi_{\nu_2} (g')$.
          Thus, it suffices to show the result when $u ( \nu_1 ) < u ( \nu_2 )$ holds.
          One has
          \begin{alignat*}{2}
            \xi_\phi (g')
             & \leq \tilde{f}_{u ( \nu_1 )}^{u ( \nu_2 )} \circ u ( \phi )
             &                                                                 & = \frac{u ( \nu_2 - \phi ) \xi_{\nu_1} (g') + u ( \phi - \nu_1 ) \xi_{\nu_2} (g')}{u (\nu_2 - \nu_1)},        \\
            \xi_{R_g \phi} (g')
             & \leq \tilde{f}_{u ( \nu_1 )}^{u ( \nu_2 )} \circ u ( R_g \phi )
             &                                                                 & = \frac{u ( \nu_2 - R_g \phi ) \xi_{\nu_1} (g') + u ( R_g \phi - \nu_1 ) \xi_{\nu_2} (g')}{u (\nu_2 - \nu_1)}
          \end{alignat*}
          by \eqref{eq:convolution-nu-inequality} and the convexity of $f$.
          Since $\nu_2 = \phi + R_g \phi - \nu_1$, one has
          \begin{align*}
            \xi_\phi (g') + \xi_{R_g \phi} (g')
             & \leq \frac{u ( 2 \nu_2 - \phi - R_g \phi ) \xi_{\nu_1} (g') + u (\phi + R_g \phi - 2 \nu_1 ) (g') \xi_{\nu_2} ( g' )}{u (\nu_2 - \nu_1)} \\
             & = \xi_{\nu_1} (g') + \xi_{\nu_2} (g').
          \end{align*}

    \item
          Since $\phi_1 * R_g \phi (g') = \phi_1 * \phi (g'g)$ holds for any $g' \in G$, one has $\xi_{R_g \phi} (g') = \xi_{\phi} (g'g)$.
          Thus, $\| \xi_{R_g \phi} \| = \| \xi_{\phi} \|$ holds and hence we obtain $2 \| \xi_\phi \| \leq \| \xi_{\nu_1} \| + \| \xi_{\nu_2} \|$ by \ref{item:S-convex-proof-xi-sum}.
          In particular, $2 S(I) \leq S(I_1') + S(I_2')$ holds by taking the upper bound on $\phi \in \mathcal{B} (I)$.
          \qedhere
  \end{enumerate}

\end{proof}

\begin{proof}[Proof of Lemma \ref{lem:S-property}]

  \begin{enumerate}
    \item
          For any $\phi \in \mathcal{B} (I)$, there is $\bar{\phi} \in \mathcal{B} (I')$ with $\phi \leq \bar{\phi}$.
          Thus, one has $S(I) \leq S(I')$ by $\xi_\phi \leq \xi_{\bar{\phi}}$.
          On the other hand, for any $\bar{\phi} \in \mathcal{B} (I')$, there is $\phi \in \mathcal{B} (I)$ with $\phi \leq \bar{\phi}$.
          Since $\xi_{\bar{\phi}} - \xi_{\phi} \leq \phi_1 * ( \bar{\phi} - \phi)$ holds, one has
          \begin{align*}
            \| \xi_{\bar{\phi}} \| - \| \xi_{\phi} \|
            \leq \| \phi_1 * ( \bar{\phi} - \phi) \|
            = I'- I
          \end{align*}
          by Lemma \ref{lem:Fubini}.
          Thus, $S(I') - S(I) \leq I' - I$ holds.
          In particular, $S$ is Lipschitz continuous.

    \item
          Let $0 \leq J_1 < J_2 \leq m (G)$.
          Since $[J_1,J_2]$ is compact, there is
          \begin{align*}
            M := \min_{J_1 \leq I \leq J_2} \hat{S}_{J_1}^{J_2} (I)
          \end{align*}
          by \ref{item:S-property-Lipschitz}.
          It suffices to show that $M=0$.
          The set
          \begin{align*}
            \mathcal{C} := \left\{ J_1 \leq I \leq J_2 \; \middle| \; \hat{S}_{J_1}^{J_2} (I) = M \right\}
          \end{align*}
          is closed by \ref{item:S-property-Lipschitz}.

          We will show that $\mathcal{C} \cap (J_1,J_2)$ is open.
          Suppose $I \in \mathcal{C}$,
          \begin{align}
            \max \{ J_1 , I^2/m(G) , 2I-J_2 \} < I_1' \leq I \label{eq:I1-I-close}
          \end{align}
          and $I_2' := 2I - I_1'$.
          It suffices to show that $\hat{S}_{J_1}^{J_2} (I_1') = \hat{S}_{J_1}^{J_2} (I_2') = M$.
          Since $J_1 < I_1' \leq I_2' < J_2$ holds by \eqref{eq:I1-I-close}, one has
          \begin{equation}
            2M \leq \hat{S}_{J_1}^{J_2} (I_1') + \hat{S}_{J_1}^{J_2} (I_2') = \tilde{S}_{J_1}^{J_2} (I_1') - S(I_1') + \tilde{S}_{J_1}^{J_2} (I_2') - S(I_2'). \label{eq:2M-leq-sum-hat}
          \end{equation}
          As
          \begin{align*}
            \tilde{S}_{J_1}^{J_2} (I_1') + \tilde{S}_{J_1}^{J_2} (I_2')
            = \frac{(2J_2 - I_1' - I_2') S(J_1) + (I_1' + I_2' - 2J_1) S(J_2)}{J_2-J_1}
            = 2 \tilde{S}_{J_1}^{J_2} (I)
          \end{align*}
          holds by $I_1' +I_2' = 2I$, one has
          \begin{equation}
            \tilde{S}_{J_1}^{J_2} (I_1') - S(I_1') + \tilde{S}_{J_1}^{J_2} (I_2') - S(I_2')
            \leq 2 ( \tilde{S}_{J_1}^{J_2} (I) - S (I) )
            = 2 \hat{S}_{J_1}^{J_2} (I) \label{eq:2M-geq-sum-hat}
          \end{equation}
          by Lemma \ref{lem:S-convex-proof} \ref{item:S-convex-proof-xi-norm}.
          Since $\hat{S}_{J_1}^{J_2} (I) = M$ holds by $I \in \mathcal{C}$, the equalities are attained in \eqref{eq:2M-leq-sum-hat} and \eqref{eq:2M-geq-sum-hat}.
          Thus, one has $\hat{S}_{J_1}^{J_2} (I_1') = \hat{S}_{J_1}^{J_2} (I_2') = M$ and hence $\mathcal{C} \cap (J_1,J_2)$ is open.

          Since $\mathcal{C} \cap (J_1,J_2)$ is clopen on the connected set $(J_1,J_2)$, one has either $\mathcal{C} \cap (J_1,J_2) = \emptyset$ or $(J_1,J_2) \subset \mathcal{C}$.
          In the case of $\mathcal{C} \cap (J_1,J_2) = \emptyset$, one has either $M = \hat{S}_{J_1}^{J_2} (J_1)$ or $M = \hat{S}_{J_1}^{J_2} (J_2)$, and hence we obtain $M = 0$ by
          \begin{equation}
            \hat{S}_{J_1}^{J_2} (J_1) = \hat{S}_{J_1}^{J_2} (J_2) = 0. \label{eq:hat-0}
          \end{equation}
          In the case of $(J_1,J_2) \subset \mathcal{C}$, one has $\mathcal{C} = [J_1,J_2]$ because $\mathcal{C}$ is closed.
          Thus, we obtain $M=0$ by \eqref{eq:hat-0}.

    \item
          Since $\xi_\phi = 0$ holds for any $\phi \in \mathcal{B} (t)$ by \eqref{eq:psi-leq-I2}, one has $S(t) = 0$.

    \item
          As $\phi_1 * \phi \leq 1$ holds by \eqref{eq:psi-leq-I1}, one has $\xi_\phi \leq \tilde{f}_0^1 \circ (\phi_1 * \phi) = (1-t) (\phi_1 * \phi)$ by the convexity of $f$.
          Since $\| \xi_\phi \| \leq (1-t) \| \phi_1 * \phi \| = (1-t) I$ holds by Lemma \ref{lem:Fubini}, we obtain $S(I) \leq (1-t) I$.
          \qedhere
  \end{enumerate}

\end{proof}

\section{Proof of Lemma \ref{lem:ft-convolution-L1-L-infinity} when \texorpdfstring{$m (G) = \infty$}{m(G)=∞}}
\label{sec:main-ft-m-infinite}

In this section, we show \eqref{eq:main-ft-S} when $m (G) = \infty$ holds.
It suffices to show
\begin{equation}
  S (I_2) \leq (1 - t)(I_2 - t) \label{eq:ft-convolution-express}
\end{equation}
for any $I_2 \geq t$ by \eqref{eq:main-ft-S}.
If $I \geq I_2$ holds, then
\begin{align*}
  S(I_2)
  \leq \tilde{S}_{t}^{I} (I_2)
  = \frac{(I - I_2) S(t) + (I_2 - t)S(I)}{I - t}
\end{align*}
holds by Lemma \ref{lem:S-property} \ref{item:S-property-convex}.
In addition, one has
\begin{align*}
  \frac{(I - I_2) S(t) + (I_2 - t)S(I)}{I - t} \leq \frac{(1 - t)(I_2 - t) I}{I - t}
\end{align*}
by Lemma \ref{lem:S-property} \ref{item:S-property-t} and \ref{item:S-property-major}.
Thus,
\begin{align*}
  S (I_2) \leq \lim_{I \to \infty} \frac{(1 - t)(I_2 - t) I}{I - t}
  = (1 - t)(I_2 - t)
\end{align*}
holds because $I_2 \leq I < \infty ( = m (G))$ is arbitrary.
Therefore, when $m (G) = \infty$ holds, we obtain \eqref{eq:ft-convolution-express} and hence Lemma \ref{lem:ft-convolution-L1-L-infinity}.

\section{Proof of Lemma \ref{lem:ft-convolution-L1-L-infinity} when \texorpdfstring{$m (G) < \infty$}{m(G)<∞}}
\label{sec:main-ft-m-finite}

In this section, we show \eqref{eq:main-ft-S} when $m (G) < \infty$ holds.
We simply write $f$ for $f_t$ in \eqref{eq:ft-define}.
It suffices to show \eqref{eq:main-ft-S} when $\| \phi_1 \| = 1$ and $\| \phi_2 \| > 0$ hold by Remark \ref{rem:convex-convolution} \ref{item:convex-convolution-normalize}.
In this case, $G_0 \subset G$ is an open subgroup because
\begin{align*}
  0 \leq 1 - t \leq m (G) - \| \phi_2 \| < m (G) = \mathrm{vol} (G_0)
\end{align*}
holds by Remark \ref{rem:convex-convolution} \ref{item:convex-convolution-connected-component} and \eqref{eq:m-large-weak} (we note that $G'$ is open if and only if $\mathrm{vol}(G') > 0$, for any closed subgroup $G' \subset G$).
Thus, we can take the Haar measure of $G_0$ to be the same as of $G$.
That is,
\begin{align}
  \| \phi \| = \sum_{g G_0 \in G / G_0} \| \phi \|_{g G_0} \label{eq:integral-decompose}
\end{align}
holds for any integrable function $\phi \colon G \to \mathbb{R}_{\geq 0}$.

We give a lemma (Lemma \ref{lem:rearrange-large}) to prove Lemma \ref{lem:ft-convolution-L1-L-infinity}.
For a measurable function $\phi \colon G \to \mathbb{R}_{\geq 0}$, we define a function $\phi^* \colon \mathbb{R} \times G/G_0 \to \{ 0,1 \}$.

\begin{definition}
  \label{def:rearrange}

  For a measurable function $\phi \colon G \to \mathbb{R}_{\geq 0}$, the function $\phi^* \colon \mathbb{R} \times G/G_0 \to \{ 0,1 \}$ is defined as
  \begin{align*}
    \phi^* (x, g G_0) := \phi_{(P(g))} (x), \quad
    P(g) := \| \phi \|_{g G_0} = \int_{g G_0}^{} \phi (g') dg',
  \end{align*}
  where $\phi_{(I)}$ is the same as in Example \ref{ex:main} \ref{item:main-psi0}.

\end{definition}

Lemma \ref{lem:ft-convolution-L1-L-infinity} is deduced by the following lemma.

\begin{lemma}
  \label{lem:rearrange-large}

  If integrable functions $\phi_1 , \phi_2 \colon G \to [0,1]$ satisfy \eqref{eq:m-large-weak}, then
  \begin{align*}
    \| f \circ ( \phi_1 * \phi_2 ) \|_{g G_0} \leq \| f \circ (\phi_1^* * \phi_2^*) \|_{\mathbb{R} \times \{ g G_0 \} }
  \end{align*}
  holds for any $g \in G$.
  In particular, one has $\| f \circ ( \phi_1 * \phi_2 ) \| \leq \| f \circ (\phi_1^* * \phi_2^*) \|$.

\end{lemma}

The estimate $\| f \circ ( \phi_1 * \phi_2 ) \| \leq \| f \circ (\phi_1^* * \phi_2^*) \|$ follows from
\begin{align*}
  \| f \circ ( \phi_1 * \phi_2 ) \|
   & = \sum_{g G_0 \in G / G_0} \| f \circ ( \phi_1 * \phi_2 ) \|_{g G_0}                               \\
   & \leq \sum_{g G_0 \in G / G_0} \| f \circ (\phi_1^* * \phi_2^*) \|_{\mathbb{R} \times \{ g G_0 \} } \\
   & = \| f \circ (\phi_1^* * \phi_2^*) \|
\end{align*}
by \eqref{eq:integral-decompose}.

\begin{example}
  \label{ex:rearrange-connected}

  If $G = G_0$, then one has $\phi_1^* = \phi_{(1)}$ and $\phi_2^* = \phi_{( \| \phi_2 \| )}$ by identifying $\mathbb{R} \times G/G_0$ with $\mathbb{R}$.
  Thus, $\| f \circ (\phi_1^* * \phi_2^*) \| = \| f \circ \psi_{1, \| \phi_2 \| } \|$ holds and hence Lemma \ref{lem:rearrange-large} is equivalent to Lemma \ref{lem:ft-convolution-L1-L-infinity} and \eqref{eq:main-ft-S} by \eqref{eq:ft-convolution-psi0}.

\end{example}

Since $m ( \mathbb{R} \times G/G_0 ) = \infty$, $\| \phi_1 \| = \| \phi_1^* \| = 1$ and $\| \phi_2 \| = \| \phi_2^* \|$, one has
\begin{align*}
  \| f \circ ( \phi_1 * \phi_2 ) \|
  \leq \| f \circ (\phi_1^* * \phi_2^*) \|
  \leq (1 - t) (\| \phi_2 \| - t)
\end{align*}
by the argument in Section \ref{sec:main-ft-m-infinite}.
Thus, we obtain Lemma \ref{lem:ft-convolution-L1-L-infinity} from Lemma \ref{lem:rearrange-large}.

Lemma \ref{lem:rearrange-large} is shown in two steps.
We will show Lemma \ref{lem:rearrange-large} in the case of $G = G_0$ (that is, $G$ is connected) in Subsection \ref{subsec:main-ft-m-finite-connected} at first, and for arbitrary $G$ in Subsection \ref{subsec:main-ft-m-finite-general} by using the case of $G = G_0$.

\subsection{Proof of Lemma \ref{lem:rearrange-large} when \texorpdfstring{$G = G_0$}{G=G0}}
\label{subsec:main-ft-m-finite-connected}

In this subsection, we show Lemma \ref{lem:rearrange-large} when $G = G_0$ holds.
It suffices to show \eqref{eq:main-ft-S} by Example \ref{ex:rearrange-connected}.
We have
\begin{align*}
  S (I)
  \leq \tilde{S}_{t}^{I'} (I)
  = \frac{(I' - I) S(t) + (I - t) S(I')}{I' - t}
\end{align*}
by Lemma \ref{lem:S-property} \ref{item:S-property-convex} for $I' := m (G) + t - 1$.
Since $G = G_0$ and $t = 1 + I' - m(G)$, it follows that
\begin{align*}
  S (I')
  = I' - t m (G)
  = ( 1 - t ) ( I' - t )
\end{align*}
by Lemma \ref{lem:m-small-counterexample} \ref{item:m-small-counterexample-ft-norm}.
Thus,
\begin{align*}
  S (I)
  \leq \frac{(I' - I) S(t) + (I - t) S(I')}{I' - t}
  = (1-t) (I-t)
\end{align*}
holds by Lemma \ref{lem:S-property} \ref{item:S-property-t}.
Therefore, we obtain \eqref{eq:main-ft-S} and Lemma \ref{lem:rearrange-large} when $G =G_0$ holds.

\subsection{Proof of Lemma \ref{lem:rearrange-large} for any \texorpdfstring{$G$}{G}}
\label{subsec:main-ft-m-finite-general}

In this subsection, we show Lemma \ref{lem:rearrange-large} for any $G$.
It suffices to show it when $g = e$ by the following lemma and remark.

\begin{lemma}
  \label{lem:rearrange-action}

  One has $(R_g \phi )^* = R_{(0, gG_0)} ( \phi^* )$ for any integrable function $\phi \colon G \to \mathbb{R}_{\geq 0}$ and $g \in G$, where $R_g \phi$ is the same as in Lemma \ref{lem:S-convex-proof}.

\end{lemma}

\begin{proof}

  Let $P(g)$ be the same as in Definition \ref{def:rearrange}.
  One has
  \begin{align*}
    ( R_g \phi )^* (x, g' G_0)
    = \phi_{(\| R_g \phi \|_{g' G_0})} (x)
    = \phi_{(P(g'g))} (x)
    = \phi^* (x, g'g G_0 )
    = R_{(0,g G_0)} ( \phi^* ) (x,g' G_0)
  \end{align*}
  for any $x \in \mathbb{R}$ and $g' \in G$.
\end{proof}

\begin{remark}
  \label{rem:rearrange-identity-suffice}

  Under the hypothesis of Lemma \ref{lem:rearrange-large},
  \begin{align*}
    \| f \circ ( \phi_1 * \phi_2 ) \|_{g G_0}
    = \| f \circ R_g ( \phi_1 * \phi_2 ) \|_{G_0}
    = \| f \circ ( \phi_1 * R_g \phi_2 ) \|_{G_0}
  \end{align*}
  and similarly
  \begin{align*}
    \| f \circ ( \phi_1^* * \phi_2^* ) \|_{ \mathbb{R} \times \{ g G_0 \} } = \| f \circ ( \phi_1^* * R_{(0, g G_0)} ( \phi_2^*) ) \|_{\mathbb{R} \times \{ G_0 \} }
  \end{align*}
  hold.
  Thus, one has
  \begin{align*}
    \| f \circ ( \phi_1^* * \phi_2^* ) \|_{ \mathbb{R} \times \{ g G_0 \} }
    = \| f \circ ( \phi_1^* * ( R_g \phi_2)^* ) \|_{\mathbb{R} \times \{ G_0 \} }
  \end{align*}
  by Lemma \ref{lem:rearrange-action}, and hence it suffices to show Lemma \ref{lem:rearrange-large} in the case of $g=e$ by replacing $\phi_2$ with $R_g \phi_2$.

\end{remark}

We show Lemma \ref{lem:rearrange-large} when $g=e$ holds.
We define
\begin{align*}
  \phi_1 *_{G_0} \phi_2 (g') := \int_{G_0}^{} \phi_1 (g'') \phi_2 (g''^{-1} g') dg''
\end{align*}
for $g' \in G_0$.
We also define integrable functions $L_{g''} \phi, R_{g''} \phi \colon G \to \mathbb{R}_{\geq 0}$ as
\begin{align*}
  L_{g''} \phi (g') := \phi (g''^{-1} g'), \quad
  R_{g''} \phi (g') := \phi (g' g'')
\end{align*}
for $g'' \in G$ and an integrable function $\phi \colon G \to \mathbb{R}_{\geq 0}$.
We give the following lemma to prove Lemma \ref{lem:rearrange-large}.

\begin{lemma}
  \label{lem:psi-decomposition}

  For any $g' \in G_0$, one has
  \begin{align*}
    \phi_1 * \phi_2 (g') = \sum_{g'' G_0 \in G / G_0}  (L_{g''} \phi_1) *_{G_0} (R_{g''} \phi_2) (g'' g' g''^{-1} ).
  \end{align*}

\end{lemma}

\begin{proof}

  Since $g'' G_0 g''^{-1} = G_0$ holds for any $g'' \in G$,
  \begin{align*}
    \phi_1 * \phi_2 (g')
     & = \int_{G}^{} \phi_1 (g''') \phi_2 (g'''^{-1} g') dg'''                                                              \\
     & = \sum_{g'' G_0 \in G / G_0} \int_{G_0}^{} \phi_1 (g''^{-1} g'''') \phi_2 (g''''^{-1} g'' g') dg''''                 \\
     & = \sum_{g'' G_0 \in G / G_0} \int_{G_0}^{} L_{g''} \phi_1 (g'''') R_{g''} \phi_2 (g''''^{-1} g'' g' g''^{-1}) dg'''' \\
     & = \sum_{g'' G_0 \in G / G_0}  (L_{g''} \phi_1) *_{G_0} (R_{g''} \phi_2) (g'' g' g''^{-1} )
  \end{align*}
  is obtained by \eqref{eq:integral-decompose}.
\end{proof}

\begin{example}
  \label{ex:rearrange-convolution}

  Let $P_1 (g'') := \| \phi_1 \|_{g'' G_0}$ and $P_2 (g'') := \| \phi_2 \|_{g'' G_0}$ for $g'' \in G$.
  The identity component of $\mathbb{R} \times G/G_0$ is $\mathbb{R} \times \{ G_0 \}$, and hence
  \begin{align*}
    \phi_1^* * \phi_2^* (x, G_0)
    = \sum_{g'' G_0 \in G / G_0} \phi_{( P_1 (g''^{-1}))} * \phi_{( P_2 (g''))} (x)
    = \sum_{g'' G_0 \in G / G_0} \psi_{P_1 (g''^{-1}), P_2 (g'')} (x)
  \end{align*}
  for any $x \in \mathbb{R}$ by Lemma \ref{lem:psi-decomposition}, where $\psi_{I_1,I_2}$ is the same as in Example \ref{ex:main} \ref{item:main-psi0}.
  Since $\psi_{P_1 (g''^{-1}), P_2 (g'')} (x)$ is an even and monotonically decreasing function on $x \geq 0$, the convolution $\phi_1^* * \phi_2^* (x, G_0)$ is also an even function with respect to $x \in \mathbb{R}$ that is monotonically decreasing on $x \geq 0$.

\end{example}

We show Lemma \ref{lem:rearrange-large} by Example \ref{ex:rearrange-connected}, Remark \ref{rem:rearrange-identity-suffice}, Lemma \ref{lem:psi-decomposition} and Example \ref{ex:rearrange-convolution}.
It suffices to show it when $g = e$ holds by Remark \ref{rem:rearrange-identity-suffice}.
We set $t_0 := \phi_1^* * \phi_2^* (0 , G_0 )$ and divide the proof into the following two cases.

\begin{enumerate}
  \item \label{item:t-small}
        The case of $t \leq t_0$.
        Since $\phi_1^* * \phi_2^* (1 + \| \phi_2 \|, G_0) = 0 \leq t \leq t_0$ holds and $\phi_1^* * \phi_2^*$ is continuous, there is $x_0 \geq 0$ with $\phi_1^* * \phi_2^* (x_0 , G_0) = t$ by the intermediate value theorem.
        We set $t_{g''} := \psi_{P_1(g''^{-1}), P_2 (g'')} (x_0)$ for $g'' \in G$.
        Then
        \begin{align*}
          t = \sum_{g'' G_0 \in G / G_0} t_{g''}
        \end{align*}
        holds by Example \ref{ex:rearrange-convolution}, and hence one has
        \begin{align*}
          \phi_1 * \phi_2 (g') - t = \sum_{g'' G_0 \in G / G_0} ( \omega_{g''} (g'' g' g''^{-1}) - t_{g''} )
        \end{align*}
        for any $g' \in G_0$ by Lemma \ref{lem:psi-decomposition}, where we write $\omega_{g''} := (L_{g''} \phi_1) *_{G_0} (R_{g''} \phi_2)$.
        Thus, one has
        \begin{align}
          \| f \circ ( \phi_1 * \phi_2 ) \|_{G_0}
          = \sum_{g'' G_0 \in G / G_0} \int_{\{ g' \in G_0 \mid \phi_1 * \phi_2 (g') \geq t \}}^{} ( \omega_{g''} (g'' g' g''^{-1}) - t_{g''} ) dg'. \label{eq:convolution-omega-represent}
        \end{align}

        \begin{lemma}
          \label{lem:omega-psi-larger}

          For any $g'' \in G$, one has
          \begin{align*}
            \int_{\{ g' \in G_0 \mid \phi_1 * \phi_2 (g') \geq t \}}^{} ( \omega_{g''} (g'' g' g''^{-1}) - t_{g''} ) dg'
            \leq \| \psi_{P_1 (g''^{-1}) , P_2 (g'')} - t_{g''} \|_{( -x_0, x_0 )}.
          \end{align*}

        \end{lemma}

        \begin{proof}

          Since $\omega_{g''} - t_{g''} \leq f_{t_{g''}} \circ \omega_{g''}$ and $g'' G_0 g''^{-1} = G_0$ hold, we have
          \begin{equation}
            \int_{\{ g' \in G_0 \mid \phi_1 * \phi_2 (g') \geq t \}}^{} ( \omega_{g''} (g'' g' g''^{-1}) - t_{g''} ) dg'
            \leq \| f_{t_{g''}} \circ \omega_{g''} \|_{G_0}. \label{eq:omega-t-ftg''-larger}
          \end{equation}
          Then $\| L_{g''} \phi_1 \|_{G_0} \leq \| \phi_1 \| = 1$ and $\| R_{g''} \phi_2 \|_{G_0} \leq \| \phi_2 \|$ hold, and hence one has
          \begin{align*}
            \| L_{g''} \phi_1 \|_{G_0} + \| R_{g''} \phi_2 \|_{G_0}
            \leq 1 + \| \phi_2 \|
            \leq m (G) + t
          \end{align*}
          by \eqref{eq:m-large-weak}.
          By Example \ref{ex:rearrange-connected} and Subsection \ref{subsec:main-ft-m-finite-connected}, we have
          \begin{align*}
            \| f_{t_{g''}} \circ \omega_{g''} \|_{G_0}
            \leq \| f_{t_{g''}} \circ \psi_{P_1 (g''^{-1}) , P_2 (g'')} \|_{\mathbb{R}}
            = \| \psi_{P_1 (g''^{-1}) , P_2 (g'')} - t_{g''} \|_{( -x_0, x_0 )},
          \end{align*}
          where the last equality is due to the monotonicity of $\psi_{P_1 (g''^{-1}) , P_2 (g'')}$ and the definition of $ t_{g''}$.
          Thus, we obtain
          \begin{align*}
            \int_{\{ g' \in G_0 \mid \phi_1 * \phi_2 (g') \geq t \}}^{} ( \omega_{g''} (g'' g' g''^{-1}) - t_{g''} ) dg'
             & \leq \| f_{t_{g''}} \circ \omega_{g''} \|_{G_0}                        \\
             & \leq \| \psi_{P_1 (g''^{-1}) , P_2 (g'')} - t_{g''} \|_{( -x_0, x_0 )}
          \end{align*}
          by \eqref{eq:omega-t-ftg''-larger}.
        \end{proof}

        Since
        \begin{align*}
          \sum_{g'' G_0 \in G / G_0} \| \psi_{P_1 (g''^{-1}) , P_2 (g'')} - t_{g''} \|_{( -x_0, x_0 )}
          = \| f \circ ( \phi_1^* * \phi_2^* ) \|_{\mathbb{R} \times \{ G_0 \} }
        \end{align*}
        holds by Example \ref{ex:rearrange-convolution}, we obtain
        \begin{align*}
          \| f \circ ( \phi_1 * \phi_2 ) \|_{G_0}
          \leq \| f \circ ( \phi_1^* * \phi_2^* ) \|_{\mathbb{R} \times \{ G_0 \} }
        \end{align*}
        by \eqref{eq:convolution-omega-represent} and Lemma \ref{lem:omega-psi-larger}.

  \item \label{item:t-large}
        The case of $t \geq t_0$.
        One has $\| f \circ ( \phi_1 * \phi_2 ) \|_{G_0} \leq \| f_{t_0} \circ ( \phi_1 * \phi_2 ) \|_{G_0}$.
        In addition,
        \begin{align*}
          \| f_{t_0} \circ ( \phi_1 * \phi_2 ) \|_{G_0} \leq \| f_{t_0} \circ ( \phi_1^* * \phi_2^* ) \|_{\mathbb{R} \times \{ G_0 \} }
        \end{align*}
        holds according to \ref{item:t-small}.
        One has $\phi_1^* * \phi_2^* (x , G_0 ) \leq t_0$ for any $x \in \mathbb{R}$ by Example \ref{ex:rearrange-convolution} and hence
        \begin{align*}
          \| f \circ ( \phi_1 * \phi_2 ) \|_{G_0}
          = \| f \circ ( \phi_1^* * \phi_2^* ) \|_{\mathbb{R} \times \{ G_0 \} }
          = 0.
        \end{align*}

\end{enumerate}

Thus, we obtain $\| f \circ ( \phi_1 * \phi_2 ) \| \leq \| f \circ ( \phi_1^* * \phi_2^* ) \|$.

\section{Proof of Theorem \ref{thm:convex-convolution-L1-L-infinity} for any convex function \texorpdfstring{$f$}{f}}
\label{sec:f-general}

In this section, we show Theorem \ref{thm:convex-convolution-L1-L-infinity} by Lemma \ref{lem:ft-convolution-L1-L-infinity}.
It suffices to show it when $\| \phi_1 \| = 1$ holds by Remark \ref{rem:convex-convolution} \ref{item:convex-convolution-normalize}.
The function $f_{(n)} \colon [0,1] \to \mathbb{R}$ is defined as
\begin{align*}
  f_{(n)} (y) := n \left( f \left( \frac{1}{n} \right) y + 2 \sum_{k = 1}^{n-1} \hat{f}_{(k-1)/n}^{(k+1)/n} \left( \frac{k}{n} \right) f_{k/n} (y) \right)
\end{align*}
for $n \in \mathbb{Z}_{\geq 1}$ and function $f \colon [0,1] \to \mathbb{R}$, where $\hat{f}_{y_1}^{y_2}$ is the same as in Definition \ref{def:convex} and $f_t$ is as defined in \eqref{eq:ft-define}.
We show the following lemma to prove Theorem \ref{thm:convex-convolution-L1-L-infinity}.

\begin{lemma}
  \label{lem:fn-approximate}

  Let $f \colon [0,1] \to \mathbb{R}$ be a convex function with $f(0) = 0$.

  \begin{enumerate}
    \item \label{item:fn-approximate-simple}
          One has $f_{(n)} (y) = \tilde{f}_{(j-1)/n}^{j/n} (y)$ for any $j = 1, \cdots , n$ and $(j-1)/n \leq y \leq j/n$, where $\tilde{f}_{y_1}^{y_2}$ is the same as in Definition \ref{def:convex}.
          In particular, $f(y) \leq f_{(n)} (y)$ holds for any $0 \leq y \leq 1$, and the equality is attained at $y = 0 , 1/n , 2/n , \cdots , 1- 1/n , 1$ (Figure \ref{fig:fn-define}).

          \begin{figure}

            \centering
            \caption{the graph of $f_{(n)}$}
            \label{fig:fn-define}

            \vspace{8pt}

            \begin{tikzpicture}
              \begin{axis}
                [
                  axis x line=center,
                  axis y line=none,
                  xmin=-0.5,
                  xmax=1.5,
                  ymin=-0.2,
                  ymax=2.5,
                  xlabel={$y$},
                  xtick={0,1},
                  xticklabels={$\displaystyle \frac{j-1}{n}$,$\displaystyle \frac{j}{n}$},
                  ytick=\empty,
                ]
                \addplot [blue, domain=-0.5:1.5, samples=50,] {2*x^2-x+1};
                \draw (1/2,1) [blue, below right] node {$f(y)$};
                \draw [red] (-1,4) -- (0,1) -- (1,2) -- (2,7);
                \draw (1/2,3/2) [red, above left] node {$f_{(n)}(y)$};
                \draw [dotted] (0,0) -- (0,1);
                \draw [dotted] (1,0) -- (1,2);
              \end{axis}
            \end{tikzpicture}

          \end{figure}
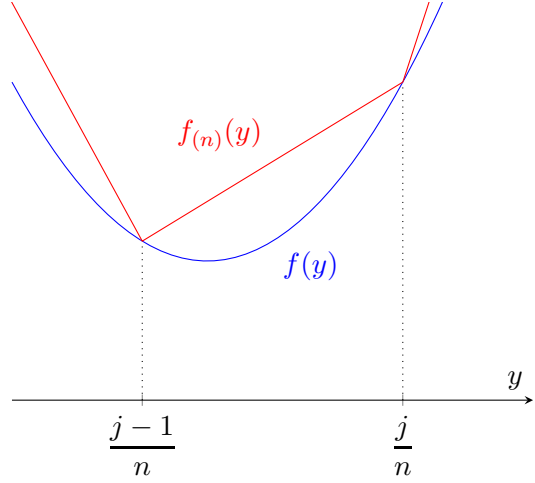

    \item \label{item:fn-approximate-f-close}
          For any $0 < y < 1$, one has
          \begin{align*}
            f_{(n)} (y) \leq f (y) + \frac{\hat{f}_{0}^{1} (y)}{4n y (1-y)}.
          \end{align*}

    \item \label{item:fn-approximate-limit}
          One has $\lim_{n \to \infty} f_{(n)} (y) = f(y)$ for any $0 \leq y \leq 1$.
          That is, $f_{(n)}$ converges pointwise to $f$.

    \item \label{item:fn-approximate-upper-bound-f1}
          One has $f_{(n)} (y) \leq f(1) y$ for any $n \in \mathbb{Z}_{\geq 1}$ and $0 \leq y \leq 1$.

    \item \label{item:fn-approximate-integral}
          For any integrable function $\psi \colon G \to [0,1]$, one has
          \begin{align*}
            \lim_{n \to \infty} \int_{G}^{} f_{(n)} \circ \psi (g) dg = \int_{G}^{} f \circ \psi (g) dg.
          \end{align*}

  \end{enumerate}

\end{lemma}

\begin{proof}

  \begin{enumerate}
    \item
          Since $(j-1)/n \leq y \leq j/n$ holds, one has
          \begin{align*}
            f_{k/n} (y) = \left\{
            \begin{aligned}
               & y - \frac{k}{n} &  & \text{if} \; k \leq j - 1, \\
               & 0               &  & \text{if} \; k \geq j
            \end{aligned}
            \right.
          \end{align*}
          for any $k = 1 , \cdots , n$.
          Then
          \begin{align*}
            \hat{f}_{(k-1)/n}^{(k+1)/n} \left( \frac{k}{n} \right)
            = \frac{1}{2} \left( f \left( \frac{k-1}{n} \right) + f \left( \frac{k+1}{n} \right) \right) - f \left( \frac{k}{n} \right)
          \end{align*}
          holds, and hence one has
          \begin{align*}
            f_{(n)} (y) = n f \left( \frac{1}{n} \right) y + \sum_{k = 1}^{j-1} \left( f \left( \frac{k-1}{n} \right) + f \left( \frac{k+1}{n} \right) - 2 f \left( \frac{k}{n} \right) \right) ( n y - k ).
          \end{align*}
          Since
          \begin{align*}
            \sum_{k = 1}^{j-1} f \left( \frac{k - 1}{n} \right) ( n y - k )
            = f \left( \frac{j - 1}{n} \right) (j - ny) + \sum_{k = 1}^{j-1} f \left( \frac{k}{n} \right) ( n y - k - 1 )
          \end{align*}
          holds by $f (0) = 0$ and
          \begin{align*}
            n f \left( \frac{1}{n} \right) y + \sum_{k = 1}^{j-1} f \left( \frac{k + 1}{n} \right) ( n y - k )
            = f \left( \frac{j}{n} \right) (ny - j + 1) + \sum_{k = 1}^{j-1} f \left( \frac{k}{n} \right) ( n y - k + 1 )
          \end{align*}
          holds, one has
          \begin{align*}
            f_{(n)} (y)
            = \tilde{f}_{(j-1)/n}^{j/n} (y) + \sum_{k = 1}^{j-1} f \left( \frac{k}{n} \right) ( n y - k - 1 + ny - k + 1 - 2 (ny - k) )
            = \tilde{f}_{(j-1)/n}^{j/n} (y).
          \end{align*}
          In particular, $f(y) \leq f_{(n)} (y)$ holds for any $0 \leq y \leq 1$ by the convexity of $f$, and the equality is attained at $y = 0 , 1/n , 2/n , \cdots , 1- 1/n , 1$.

    \item
          It suffices to show the result when $j = 1 , \cdots , n$ and $(j-1)/n \leq y \leq j/n$ hold.
          One has
          \begin{equation}
            f_{(n)} (y)
            \leq \tilde{f}_{(j-1)/n}^{j/n} (y)
            = (j - ny) f \left( \frac{j-1}{n} \right) + (ny - j + 1) f \left( \frac{j}{n} \right) \label{eq:fn-represent}
          \end{equation}
          by \ref{item:fn-approximate-simple}.
          Then
          \begin{align*}
            f \left( \frac{j - 1}{n} \right)
            \leq \tilde{f}_{0}^{y} \left( \frac{j - 1}{n} \right)
            = \frac{(j-1) f(y)}{n y}
          \end{align*}
          holds by $0 \leq (j-1)/n \leq y$, $f (0) = 0$ and the convexity of $f$, and similarly
          \begin{align*}
            f \left( \frac{j}{n} \right)
            \leq \tilde{f}_{y}^{1} \left( \frac{j}{n} \right)
            = \frac{(n-j) f(y) + (j-ny) f(1)}{n (1-y)}
          \end{align*}
          holds by $y \leq j/n \leq 1$.
          Thus, one has
          \begin{align*}
             & \quad f_{(n)} (y)                                                                                                                                       \\
             & \leq \frac{(j - ny) (j-1) f(y)}{ny} + \frac{(ny - j + 1) ((n-j) f(y) + (j-ny) f(1))}{n (1-y)}                                                           \\
             & = \frac{f (y) - f(1) y}{ny (1-y)} \cdot j^2 + \frac{(2ny + 1) (f (1) y - f (y))}{n y (1-y)} \cdot j + f (y) + \frac{(n y + 1) ( f(y) - y f (1))}{1 - y} \\
             & = f (y) + \frac{f (y) - f(1) y}{ny (1-y)} \left( \left( j - ny - \frac{1}{2} \right)^2 - \frac{1}{4} \right)                                            \\
             & = f (y) + \frac{\hat{f}_{0}^{1} (y) (1 - (2ny + 1 - 2j)^2 )}{4ny(1-y)}.
          \end{align*}
          Since $\hat{f}_{0}^{1} (y) \geq 0$ holds by the convexity of $f$, we obtain
          \begin{align*}
            f_{(n)} (y)
            \leq f (y) + \frac{\hat{f}_{0}^{1} (y) (1 - (2ny + 1 - 2j)^2 )}{4ny(1-y)}
            \leq f (y) + \frac{\hat{f}_{0}^{1} (y)}{4ny(1-y)}.
          \end{align*}

    \item
          If $y = 0,1$ holds, then one has $f_{(n)} (y) = f (y)$ by \ref{item:fn-approximate-simple}.
          Thus, it suffices to show the result when $0 < y < 1$.
          One has
          \begin{align*}
            f (y) \leq f_{(n)} (y) \leq f (y) + \frac{a}{n}, \quad
            a := \frac{\hat{f}_{0}^{1} (y)}{4y(1-y)}
          \end{align*}
          by \ref{item:fn-approximate-simple} and \ref{item:fn-approximate-f-close}.
          Since
          \begin{align*}
            \lim_{n \to \infty} \left( f (y) + \frac{a}{n} \right)
            = f (y),
          \end{align*}
          we obtain
          \begin{align*}
            \lim_{n \to \infty} f_{(n)} (y) = f(y).
          \end{align*}

    \item
          It suffices to show the result when $j = 1 , \cdots , n$ and $(j-1)/n \leq y \leq j/n$ hold.
          One has
          \begin{align*}
            f_{(n)} (y)
            \leq (j - ny) f \left( \frac{j-1}{n} \right) + (ny - j + 1) f \left( \frac{j}{n} \right)
          \end{align*}
          by \eqref{eq:fn-represent}.
          In addition,
          \begin{align*}
            f \left( \frac{j - 1}{n} \right)
            \leq \tilde{f}_{0}^{1} \left( \frac{j-1}{n} \right)
             & = \frac{(j-1) f(1)}{n} , &  & \text{and} &
            f \left( \frac{j}{n} \right)
            \leq \tilde{f}_{0}^{1} \left( \frac{j}{n} \right)
             & = \frac{j f(1)}{n}
          \end{align*}
          hold by $f(0) =0$ and the convexity of $f$.
          Thus, one has
          \begin{align*}
            f_{(n)} (y) \leq \frac{(j - ny) (j-1) f(1)}{n} + \frac{(ny - j + 1) j f(1)}{n}
            = f(1) y.
          \end{align*}

    \item
          Since $f_{(n)} \geq f$ by \ref{item:fn-approximate-simple}, one has
          \begin{equation}
            \liminf_{n \to \infty} \int_{G}^{} f_{(n)} \circ \psi (g) dg \geq \int_{G}^{} f \circ \psi (g) dg. \label{eq:fn-integral-liminf}
          \end{equation}
          One has $f_{(n)} \circ \psi (g) \leq f (1) \psi (g)$ for any $n \in \mathbb{Z}_{\geq 1}$ by \ref{item:fn-approximate-upper-bound-f1}.
          Since $\psi$ is integrable,
          \begin{equation}
            \limsup_{n \to \infty} \int_{G}^{} f_{(n)} \circ \psi (g) dg \leq \int_{G}^{} f \circ \psi (g) dg \label{eq:fn-integral-limsup}
          \end{equation}
          holds by applying Fatou's lemma to $f(1) \psi - f_{(n)} \circ \psi$.
          Thus, we obtain
          \begin{align*}
            \lim_{n \to \infty} \int_{G}^{} f_{(n)} \circ \psi (g) dg = \int_{G}^{} f \circ \psi (g) dg
          \end{align*}
          by \eqref{eq:fn-integral-liminf} and \eqref{eq:fn-integral-limsup}.
          \qedhere
  \end{enumerate}

\end{proof}

\begin{proof}[Proof of Theorem \ref{thm:convex-convolution-L1-L-infinity}]

  It suffices to show the result when $\| \phi_1 \| = 1$ holds by Remark \ref{rem:convex-convolution} \ref{item:convex-convolution-normalize}.
  Since
  \begin{align}
    \| \phi_1 * \phi_2 \| = \| \psi_{1,\| \phi_2 \|} \| = \| \phi_2 \| < \infty \label{eq:Fubini-apply}
  \end{align}
  holds by Lemma \ref{lem:Fubini}, one has
  \begin{align}
    \int_{G}^{} f \circ ( \phi_1 * \phi_2 ) (g) dg
    = \lim_{n \to \infty} \int_{G}^{} f_{(n)} \circ ( \phi_1 * \phi_2 ) (g) dg \label{eq:composition-express}
  \end{align}
  by Lemma \ref{lem:fn-approximate} \ref{item:fn-approximate-integral}.
  By the definition of $f_{(n)}$,
  \begin{align}
    \int_{G}^{} f_{(n)} \circ ( \phi_1 * \phi_2 ) (g) dg
     & = n \left( f \left( \frac{1}{n} \right) \| \phi_1 * \phi_2 \| + 2 \sum_{k = 1}^{n-1} \hat{f}_{(k-1)/n}^{(k+1)/n} \left( \frac{k}{n} \right) \| f_{k/n} \circ ( \phi_1 * \phi_2 ) \| \right) \label{eq:fn-phi1-phi2}
  \end{align}
  holds.
  Since \eqref{eq:ft-convolution-psi0} holds by Lemma \ref{lem:ft-convolution-L1-L-infinity} and $\hat{f}_{(k-1)/n}^{(k+1)/n} ( k/n ) \geq 0$ holds by the convexity of $f$, one has
  \begin{align*}
    \hat{f}_{(k-1)/n}^{(k+1)/n} \left( \frac{k}{n} \right) \| f_{k/n} \circ ( \phi_1 * \phi_2 ) \|
    \leq \hat{f}_{(k-1)/n}^{(k+1)/n} \left( \frac{k}{n} \right) \| f_{k/n} \circ \psi_{1,\| \phi_2 \|} \|
  \end{align*}
  for any $k = 1 , \cdots , n-1$.
  Thus,
  \begin{align}
    \int_{G}^{} f_{(n)} \circ ( \phi_1 * \phi_2 ) (g) dg
     & \leq n \left( f \left( \frac{1}{n} \right) \| \psi_{1,\| \phi_2 \|} \| + 2 \sum_{k = 1}^{n-1} \hat{f}_{(k-1)/n}^{(k+1)/n} \left( \frac{k}{n} \right) \| f_{k/n} \circ \psi_{1,\| \phi_2 \|} \| \right) \notag \\
     & = \int_{\mathbb{R}}^{} f_{(n)} \circ \psi_{1,\| \phi_2 \|} (x) dx \label{eq:express-main-ft}
  \end{align}
  holds by \eqref{eq:Fubini-apply} and \eqref{eq:fn-phi1-phi2}.
  Since
  \begin{align}
    \lim_{n \to \infty} \int_{\mathbb{R}}^{} f_{(n)} \circ \psi_{1,\| \phi_2 \|} (x) dx
    = \int_{\mathbb{R}}^{} f \circ \psi_{1,\| \phi_2 \|} (x) dx \label{eq:fn-approximate-tilde}
  \end{align}
  holds by Lemma \ref{lem:fn-approximate} \ref{item:fn-approximate-integral}, we obtain
  \begin{align*}
    \int_{G}^{} f \circ ( \phi_1 * \phi_2 ) (g) dg
    \leq \int_{\mathbb{R}}^{} f \circ \psi_{1,\| \phi_2 \|} (x) dx
    = 2 \int_{0}^{1} f(y) dy + ( \| \phi_2 \| - 1 ) f( 1 )
  \end{align*}
  by \eqref{eq:tilde-culculation}, \eqref{eq:composition-express}, \eqref{eq:express-main-ft}, and \eqref{eq:fn-approximate-tilde}.
\end{proof}

\section*{Acknowledgement}

The author would like to thank his advisor Toshiyuki Kobayashi for his support.
The author is also grateful to Yuichiro Tanaka and Toshihisa Kubo for helpful comments.
The author also would like to thank the reviewers for the careful comments.

\section*{Funding}

This work was supported by JSPS KAKENHI Grant Number JP19J22628 and Leading Graduate Course for Frontiers of Mathematical Sciences and Physics (FMSP).

\printbibliography

\noindent
Takashi Satomi: Graduate School of Mathematical Sciences, The University of Tokyo, 3-8-1 Komaba Meguro-ku Tokyo 153-8914, Japan.

\noindent
E-mail: tsatomi@ms.u-tokyo.ac.jp

\end{document}